\documentclass[10pt,a4paper]{amsart} 
\usepackage{geometry}
\usepackage{tabls}
\usepackage{url}
\usepackage[utf8]{inputenc}
\usepackage{amsfonts}
\usepackage{amssymb}
\usepackage{mathrsfs}
\usepackage{amsmath}
\usepackage{amscd}
\usepackage{fancyhdr}
\usepackage{enumerate}
\usepackage{paralist}
\usepackage{xypic}
\usepackage{hyperref}
\usepackage{graphicx}
\usepackage{cite}
\usepackage{lipsum}
\usepackage{footmisc}
\usepackage[all,cmtip]{xy}

\allowdisplaybreaks

\numberwithin{equation}{section}
\theoremstyle{plain}
\newtheorem{theorem}{Theorem}[section]
\newtheorem{lemma}[theorem]{Lemma}

\numberwithin{equation}{section}

\def\a{\alpha}
\def\b{\beta}

\def\L{\Lambda}

\def\W{\Omega}

\def\ni{\noindent}

\def\rank{{\rm rank\/}}

\def\ind{{\rm ind\/}}
\def\Stab{{\rm Stab\/}}

\def\End{\mathrm{End}}

\def\R{\mathbb{R} }

\def\fg{\mathfrak{g}}

\begin{document}
	
	\title{Transversality of the perturbed reduced  Vafa-Witten moduli spaces on 4-manifolds}

	
	\author{Ren Guan}
	\address{School of Mathematics and Statistics, 
		Jiangsu Normal University, Xuzhou 221100, China}
	\email{guanren@jsnu.edu.cn}

	\begin{abstract}
		Previously we finish the establishment of the transversality of the general part of the Vafa-Witten moduli spaces,	in this paper, we deal with the rest, i.e., the reduced part. We consider Vafa-Witten equation on closed, oriented and smooth Riemann 4-manifolds with $C\equiv0$,  and construct  perturbation to establish the transversality of the perturbed equation. Then we show that for a generic choice of the perturbation terms, the moduli space of solutions to the perturbed reduced Vafa-Witten equation for the structure group $SU(2)$ or $SO(3)$ on a closed 4-manifold is a smooth manifold of dimension zero.

	\end{abstract}
	
	\subjclass{58D27, 53C07, 81T13}
	
	\keywords{reduced Vafa-Witten equations, Vafa-Witten moduli spaces,  transversality, 4-manifolds}
	
	\maketitle
	
	\tableofcontents
	
	\section{Introduction}
	
	Vafa-Witten equations are first proposed by Vafa and Witten  for studying an $N=4$ topologically twisted supersymmetric Yang-Mills theory on 4-manifolds \cite{VW}. They require a triplet $(A,B,C)$ to satisfy
	\begin{equation}\left\{
		\begin{aligned}
			&d_A^*B +d_A C=0,\\&F_A^++\frac{1}{8}[B\centerdot B]+\frac{1}{2}[B,C]=0,  \\
		\end{aligned}\right. \label{aa}
	\end{equation}
	where $A$ is a connection on a principal $G$-bundle $P$ over a smooth 4-manifold $X$, $B$ is a self-dual 2-form with values in the adjoint bundle $\fg_P$, and $C$ is a section of the adjoint bundle $\fg_P$. 
	
	Previously, in \cite{G} we define the \emph{Vafa-Witten moduli space} $\mathcal{M}_{VW}$ as the set of solutions to the equation \eqref{aa} modulo the gauge transformation, and then we divide it into three parts:
	$$\mathcal{M}_{VW}=\mathcal{M}_{\text{asd}}\cup\mathcal{M}_{RVW}\cup\widetilde{\mathcal{M}}_{VW}$$
	where
	\begin{align*}
		\mathcal{M}_{\text{asd}}&:=\{[A,B,C]\in\mathcal{M}_{VW}:\text{$B\equiv 0$ and $C\equiv0$ on $X$}\},\\
		\mathcal{M}_{RVW}&:=\{[A,B,C]\in\mathcal{M}_{VW}:\text{$B\not\equiv 0$ and $C\equiv0$ on $X$}\},\\
		\widetilde{\mathcal{M}}_{VW}&:=\{[A,B,C]\in\mathcal{M}_{VW}:\text{$C\not\equiv0$ on $X$}\}.
	\end{align*}
	$\mathcal{M}_{\text{asd}}$ is the well-known moduli space of the anti-self-dual (ASD) connections, $\mathcal{M}_{RVW}$ is called the \emph{reduced} part of $\mathcal{M}_{VW}$ and $\widetilde{\mathcal{M}}_{VW}$ the \emph{general} part.

	In  \cite{Tan}, by applying the method of \cite{Fe}, Tanaka constructs suitable perturbations of Vafa-Witten equations on closed symplectic 4-manifolds, then he shows that if the structure group is $SU(2)$ or $SO(3)$, the moduli spaces are smooth manifolds of dimension zero for generic choices of the perturbation parameters. On closed 4-manifolds, we studied the transversality of the full rank part of the Vafa-Witten moduli spaces \cite{DG}, and we also establish the transversality of the general part $\widetilde{\mathcal{M}}_{VW}$ \cite{G}. In this paper, by applying similar methods used in \cite{G}, we deal with the rest, the reduced part $\mathcal{M}_{RVW}$. We construct a perturbation of \eqref{aa} and establish the perturbed  equation on closed, oriented and smooth Riemannian 4-manifolds at the solutions $(A,B,C)$ where $A$ is irreducible, $B\not\equiv 0$ and $C\equiv 0$, then we show that the dimensions of this part of the moduli spaces are zero for a generic choice of the perturbation parameters. Our main result is the following theorem.
	
	\begin{theorem}\label{main}Let $(X,g)$ be a closed, oriented and smooth Riemannian 4-manifold, $r>k\geq3$ are two integers. Then there is a first-category subset $\mathcal{T}^r_{fc}\subset\mathcal{T}^r:=C^r(GL(\Lambda^{1}))\times C^r(GL(\Lambda^{2,+}))\times C^r(GL(\Lambda^{2,+}))\times C^r(X,\Lambda^{1})\times C^r(X,\Lambda^{1})$ such that for all $(\tau^1,\tau^2,\tau^3,\theta_1,\theta_2)$ in $\mathcal{T}^r-\mathcal{T}^r_{fc}$, the moduli space of the $L_k^2$ solutions $[A,B,C]$  to the equation
		\begin{equation}\left\{
			\begin{aligned}
				&d_A^*B+d_AC+\tau^1\big(B\centerdot\theta_1+[B\centerdot B]\centerdot\theta_2\big)=0\\
				&F_A^++\frac{1}{8}\tau^2[B\centerdot B]+\frac{1}{2}[B,C]+\tau^3B=0\\
			\end{aligned}\right.\end{equation}
		with $A$ irreducible, $B\not\equiv0$ and $C\equiv0$ is a smooth manifold of dimension zero.
	\end{theorem}

	Note that the Lie algebra of $SU(2)$ and $SO(3)$ are isomorphic: $\mathfrak{su}(2)\cong\mathfrak{so}(3)$, so in the following, we use $\mathfrak{su}(2)$ instead of $\mathfrak{so}(3)$.
	
	\section{Vafa-Witten equations and their perturbations}
	
	In this section, we state the general set-up of Vafa-Witten equations and construct a perturbation for getting the transversality result. 
	
	\subsection{The Vafa-Witten map}
	
	Let $(X,g)$ be a closed, oriented and smooth Riemannian 4-manifold, $P\to X$ a principle $SU(2)$- or $SO(3)$-bundle on $X$ and $\mathcal{A}_k (P)$ the space of $L_k^2$ connections where $k\geq 3$ is an integer. $\mathcal{A}_k (P)$ is an affine space, so for any fixed $L_k^2$ connection $A_0$, we have
	$$\mathcal{A}_k (P)=A_0+L_k^2(X,\mathfrak{su}(2)_P\otimes\Lambda^1).$$
	The \emph{configuration spaces} and \emph{target spaces} are defined by
	$$\begin{aligned}
		\mathcal{C}_k(P)&:=\mathcal{A}_k (P)\times L_k^2(X,\mathfrak{su}(2)_P\otimes\Lambda^{2,+})\oplus L_k^2(X, \mathfrak{su}(2)_P),\\
		\mathcal{C}_{k-1}'(P)&:=L_{k-1}^2(X,\mathfrak{su}(2)_P\otimes\Lambda^1)\oplus L_{k-1}^2(X,\mathfrak{su}(2)_P\otimes\Lambda^{2,+}).
	\end{aligned}$$
	It's easy to see that $\mathcal{C}_k(P)$ is also an affine vector space, which means for any fixed $L_k^2$ triple $(A_0,B_0,C_0)\in\mathcal{C}_k(P)$, we have
	$$\mathcal{C}_k(P)=(A_0,B_0,C_0)+L_k^2(X,\mathfrak{su}(2)_P\otimes\Lambda^1)\oplus L_k^2(X,\mathfrak{su}(2)_P\otimes\Lambda^{2,+})\oplus
	L_k^2(X,\mathfrak{su}(2)_P).$$
	
	The \emph{Vafa-Witten map} is defined by
	\begin{equation}\begin{aligned}
			\mathcal{VW}&:\mathcal{C}_k(P)\to\mathcal{C}_{k-1}'(P),\\
			\mathcal{VW}(A,B,C)&:=\begin{pmatrix}
				d_A^*B+d_AC\\
				F_A^++\frac{1}{8}[B\centerdot B]+\frac{1}{2}[B,C]\\
			\end{pmatrix}.
		\end{aligned} \label{cc}
	\end{equation}
	The action of an $L_{k+1}^2$ gauge transformation $\zeta\in\mathcal{G}_{k+1}(P)$ on any $(A,B,C)\in\mathcal{C}_k(P)$ is given by
	$$\zeta\cdot(A,B,C):=(A-(d_A\zeta^{-1})\zeta,\zeta^{-1}B\zeta,\zeta^{-1}C\zeta),$$
	and a connection $A$ is called \emph{irreducible} if $\Stab(A):=\{\zeta\in\mathcal{G}_{k+1}(P):\zeta\cdot A=A\}=Z(G)$, the center of $G$. 
	
	It's not hard to see that
	$$\mathcal{VW}(\zeta\cdot(A,B,C))=\begin{pmatrix}
		\zeta^{-1}(d_A^*B+d_AC)\zeta\\
		\zeta^{-1}(F_A^++\frac{1}{8}[B\centerdot B]+\frac{1}{2}[B,C])\zeta\\
	\end{pmatrix},$$
	so the map $\mathcal{VW}$ is gauge-equivariant (cf.  \cite[\S 3.2.1]{Ma}). Define the quotient space $\mathcal{B}_k (P):=\mathcal{C}_k (P)/\mathcal{G}_{k+1}(P)$ (cf. \cite[Propostion 2.8]{FL},  \cite[Theorem 3.2.3]{Ma}).
	
	Denote by $\mathcal{C}_k^\diamond (P) \subset\mathcal{C}_k(P)$ the triples $(A,B,C)$ with $A$ irreducible, $B\not\equiv 0$ and $C\equiv 0$. The slice theorem (cf.  \cite[Propostion 2.8]{FL} \cite[Theorem 3.2.3]{Ma}) implies that the quotient space $\mathcal{B}_k^\diamond (P):=\mathcal{C}_k^\diamond (P)/\mathcal{G}_{k+1}(P)\subset\mathcal{B}_k(P)$ is an open and smooth Hilbert submanifold of $\mathcal{B}_k(P)$. 
	
	\subsection{The perturbed reduced Vafa-Witten map}
	
	We now construct the \emph{perturbed reduced Vafa-Witten map} as follows:
	\begin{equation}\begin{aligned}
			\mathcal{RVW}&:\mathcal{T}^r\times\mathcal{C}_k(P)\to\mathcal{C}_{k-1}'(P),\\
			\mathcal{RVW}(\tau^1,\tau^2,\tau^3,\theta_1,\theta_2,A,B,C)&:=\begin{pmatrix}
				d_A^*B+d_AC+\tau^1(B\centerdot\theta_1+[B\centerdot B]\centerdot\theta_2)\\
				F_A^++\frac{1}{8}\tau^2[B\centerdot B]+\frac{1}{2}[B,C]+\tau^3B
			\end{pmatrix}, \label{cd}
		\end{aligned}
	\end{equation}
	where $\mathcal{T}^r:=C^r(GL(\Lambda^{1}))\times C^r(GL(\Lambda^{2,+}))\times C^r(GL(\Lambda^{2,+}))\times C^r(X,\Lambda^{1})\times C^r(X,\Lambda^{1})$ denotes the Banach manifold of $C^r$ perturbation parameters $(\tau^1,\tau^2,\tau^3,\theta_1,\theta_2)$ (with $r$ large enough, say $r>k$). The gauge group $\mathcal{G}_{k+1}(P)$ acts trivially on the space of perturbations $\mathcal{T}^r$, so the map $\mathcal{RVW}$ is also gauge-equivariant, and $\mathcal{PM}_k(P):=\mathcal{RVW}^{-1}(0)/\mathcal{G}_{k+1}(P)\subset\mathcal{T}^r\times\mathcal{B}_k(P)$ is the parametrized moduli space of the perturbed reduced Vafa-Witten equation. Denote $\mathcal{PM}_k^\diamond(P):=\mathcal{PM}_k(P)\cap(\mathcal{T}^r\times\mathcal{B}_k^\diamond(P))$, then the gauge-equivariant map $\mathcal{RVW}$ defines a section of a Banach vector bundle $\overline{\mathcal{E}}_k$ over $\mathcal{T}^r\times\mathcal{B}_k^\diamond(P)$ with total space $\overline{\mathcal{E}}_k:=(\mathcal{T}^r\times\mathcal{C}_k^\diamond(P))\times_{\mathcal{G}_{k+1}(P)}\mathcal{C}'_{k-1}(P)$. In particular, the parametrized moduli space $\mathcal{PM}_k^\diamond(P)$ is the zero set of the section $\mathcal{RVW}(\cdot,\cdot)$ of the vector bundle $\overline{\mathcal{E}}_k$ over $\mathcal{T}^r\times\mathcal{B}_k^\diamond(P)$. 
	
	\subsection{The Kuranishi complex}

	For each chosen perturbation parameters $(\tau^1,\tau^2,\tau^3,\theta_1,\theta_2)$, the associated Kuranishi complex of $\mathcal{RVW}$ is
	$$\begin{aligned}
		0\to L^2_{k+1}(X,\Lambda^0\otimes\mathfrak{su}(2)_P)&\xrightarrow{d^0_{(A,B,C)}}L^2_k (X,\mathfrak{su}(2)_P\otimes\Lambda^1)\oplus L^2_k(X,\mathfrak{su}(2)_P\otimes\Lambda^{2,+})\oplus\\
		L^2_k(X,\Lambda^0\otimes\mathfrak{su}(2)_P)&\xrightarrow{d^1_{(A,B,C)}} L^2_{k-1}(X,\mathfrak{su}(2)_P\otimes\Lambda^1)\oplus L^2_{k-1}(X,\mathfrak{su}(2)_P\otimes\Lambda^{2,+})\to 0,
	\end{aligned}$$
	where
	\begin{equation}
		d^0_{(A,B,C)}(\xi)=(d_A\xi, [B, \xi],[C, \xi]) 
		\label{cf} \end{equation}
	is the linearization of the action of gauge group,
	\begin{equation} \begin{aligned}
			&d^1_{(A,B,C)}(a,b,c)=(D\mathcal{RVW})_{(A,B,C)}(a,b,c) \\
			&=\begin{pmatrix}
				d^*_Ab+d_Ac-[B\centerdot a]-[C,a]+\tau^1(b\centerdot\theta_1+2[b\centerdot B]\centerdot\theta_2)\\
				d_A^+a+\frac{1}{4}\tau^2[b\centerdot B] +\frac12 [b,C] +\frac12 [B,c]+\tau^3 b
			\end{pmatrix}
		\end{aligned} \label{cg}
	\end{equation}
	is the linearization of the perturbed reduced Vafa-Witten maps. Also we have (cf. \cite{Fe,G, Ma})
	$$\begin{aligned}
		d^1_{(A,B,C)}\circ d^0_{(A,B,C)}(\xi)=\begin{pmatrix}
			[\xi,d_A^*B+d_AC+\tau^1(B\centerdot\theta_1+[B\centerdot B]\centerdot\theta_2)]\\
			[\xi,F_A^++\frac{1}{8}\tau^2[B\centerdot B]+\frac12 [B,C]+\tau^3B]\end{pmatrix}.
	\end{aligned}$$
	Therefore $d^1_{(A,B,C)}\circ d^0_{(A,B,C)}=0$ if and only if $\mathcal{RVW}(\tau^1,\tau^2,\tau^3,\theta_1,\theta_2,A,B,C)=0$, i.e. the sequence is a complex if and only if $[A,B,C]\in\mathcal{M}_{RVW}(\tau^1,\tau^2,\tau^3,\theta_1,\theta_2)$, where $\mathcal{M}_{RVW}(\tau^1,\tau^2,\tau^3,\theta_1,\theta_2):=\{[A,B,C]\in\mathcal{B}_k(P):\mathcal{RVW}(\tau^1,\tau^2,\tau^3,\theta_1,\theta_2,A,B,C)=0\}$.
	
	The $L^2$ adjoint of $d^0_{(A,B,C)}$ is (cf. \cite{G, Ma})
	$$d^{0,*}_{(A,B,C)}(a,b,c)=d_A^*a+[b\cdot B]+[c,C],$$
	and the combined opertor
	$$\begin{aligned}
		\mathcal{D}_{(A,B,C)}&=d^{1}_{(A,B,C)}+ d^{0,*}_{(A,B,C)}:\begin{matrix} L^2_k(X,\mathfrak{su}(2)_P\otimes\Lambda^1) \\ \oplus \\ L^2_k(X,\mathfrak{su}(2)_P\otimes\Lambda^{2,+}) \\ \oplus \\ L^2_k (X, \mathfrak{su}(2)_P)\end{matrix}\to\begin{matrix} L^2_{k-1}(X,\mathfrak{su}(2)_P\otimes\Lambda^1) \\ \oplus \\ L^2_{k-1}(X,\mathfrak{su}(2)_P\otimes\Lambda^{2,+}) \\ \oplus \\ L^2_{k-1}(X,\mathfrak{su}(2)_P)\end{matrix}
	\end{aligned}$$
	differs from the following operator 
	\begin{equation} \mathcal{D}_A =\left(
		\begin{array}{ccc}
			0 & d_A^* & d_A \\
			d_A^+ & 0 & 0 \\
			d_A^* & 0 & 0 \\
		\end{array}
		\right):\begin{pmatrix} a \\ b \\ c \end{pmatrix} \mapsto
		\begin{pmatrix} d_A^* b +d_A c \\ d_A^+ a \\ d_A^* a \end{pmatrix}
		\label{cj} \end{equation}
	by zeroth-order terms. By the Sobolev multiplication theorem and the Rellich embedding theorem, there is a contiuous Sobolev multiplication map $L_k^2\times L_k^2\to L_k^2$ and the inclusion $L_k^2\subset L_{k-1}^2$ is compact when $k\geq 3$, we have $\ind(\mathcal{D}_{(A,B,C)})=\ind(\mathcal{D}_A)=0$, the second equality is due to the self-adjointness of the elliptic operato $\mathcal{D}_A$. 
	
	It follows that the above complex is an elliptic deformation complex for the perturbed reduced Vafa-Witten equation with cohomology groups
	$$H_{(A,B,C)}^0:=\mathrm{Ker}d^{0}_{(A,B,C)},H_{(A,B,C)}^1:=\mathrm{Ker}d^{1}_{(A,B,C)}/\mathrm{Im}~d^{0}_{(A,B,C)},H_{(A,B,C)}^2:=\mathrm{Coker}d^{1}_{(A,B,C)}.$$
	Similarly, $H_{(A,B,C)}^0$ is the Lie algebra of the stabilizer of the triple $(A,B,C)$,  and $H_{(A,B,C)}^0=0$ if the stabilizer of $(A,B,C)$ is $Z(G)$, which means $A$ is irreducible, $H_{(A,B,C)}^1$ is the tangent space $T_{[A,B,C]}\mathcal{M}_{RVW}(\tau^1,\tau^2,\tau^3,\theta_1,\theta_2)$ and if  $\mathrm{Coker}(D\mathcal{RVW})_{(A,B,C)}=0$, then $H_{(A,B,C)}^2=0$ and $[A,B,C]$ is a regular point of $\mathcal{M}_{RVW}(\tau^1,\tau^2,\tau^3,\theta_1,\theta_2)$. We will prove later that every point of the reduced part of the moduli space $\mathcal{M}_{RVW}^*(\tau^1,\tau^2,\tau^3,\theta_1,\theta_2):=\mathcal{M}_{RVW}(\tau^1,\tau^2,\tau^3,\theta_1,\theta_2)\cap\mathcal{B}_k^\diamond(P)$ is regular and $\mathcal{M}_{RVW}^*(\tau^1,\tau^2,\tau^3,\theta_1,\theta_2)$ is a smooth manifold of dimension $\ind(\mathcal{D}_A)=0$.
	
	\section{Quadratic expansion of the perturbed reduced  Vafa-Witten map}
	
	We now find the quadratic expansion of $\mathcal{RVW}(\tau^1,\tau^2,\tau^3,\theta,\gamma,A,B,C)$ with a fixed perturbation parameter to build the regularity result (see  \cite[\S3]{FL} for the perturbed $PU(2)$-monopole equations and  \cite[\S3.2.2]{Ma} for Vafa-Witten equations). 
	Let $(A_0,B_0,C_0)\in\mathcal{C}_k(P)$ be a fixed smooth triple and $(a,b,c)\in L_k^2(X,\mathfrak{su}(2)_P\otimes\Lambda^1)\oplus L_k^2(X,\mathfrak{su}(2)_P\otimes\Lambda^{2,+})\oplus L_k^2(X,\mathfrak{su}(2)_P)$, then
	\begin{align*}
		&\mathcal{RVW}(\tau^1,\tau^2,\tau^3,\theta_1,\theta_2,A_0+a,B_0+b,C_0+c)\\
		=&\begin{pmatrix}d_{A_0+a}^*(B_0+b)+d_{A_0+a}(C_0+c)+\tau^1((B_0+b)\centerdot\theta_1+[B_0+b, B_0+b]\centerdot\theta_2)\\
			F_{A_0+a}^++\frac{1}{8}\tau^2[B_0+b\centerdot B_0+b]+\frac{1}{2}[B_0+b,C_0+c]+\tau^3(B_0+b)\end{pmatrix}\\
		=&\mathcal{RVW}(\tau^1,\tau^2,\tau^3,\theta_1,\theta_2,A_0,B_0,C_0)+d_{(A_0,B_0,C_0)}^1(a,b,c)\\
		&+\begin{pmatrix}-[a\centerdot b]+[a,c]+\tau^1([b\centerdot b]\centerdot\theta_2)\\
			\frac{1}{2}[a\wedge a]^++\frac{1}{8}\tau^2[b\centerdot b]+\frac{1}{2}[b,c]\end{pmatrix} \\
		=&\mathcal{RVW}(\tau^1,\tau^2,\tau^3,\theta_1,\theta_2,A_0,B_0,C_0)+d_{(A_0,B_0,C_0)}^1(a,b,c)+\{(a,b,c),(a,b,c)\},
	\end{align*}
	where
	$$\begin{aligned}
		\{(a,b,c),(a,b,c)\}:=\begin{pmatrix}-[a\centerdot b]+[a,c]+\tau^1([b\centerdot b]\centerdot\theta_2)\\
			\frac{1}{2}[a\wedge a]^++\frac{1}{8}\tau^2[b\centerdot b]+\frac{1}{2}[b,c]\end{pmatrix}.
	\end{aligned}$$
	Given $(A_0,B_0,C_0)\in\mathcal{C}_k(P)$ and $(u_0,v_0)\in\mathcal{C}_{k-1}'(P)$, consider the inhomogeneous equation
	\begin{equation}\label{pvw}
		\mathcal{RVW}(\tau^1,\tau^2,\tau^3,\theta_1,\theta_2,A_0+a,B_0+b,C_0+c)=(u_0,v_0)
	\end{equation}
	for the triplets $(a,b,c)$. To make the equation elliptic, we impose the gauge-fixing condition
	\begin{equation}\label{gf}
		d^{0,*}_{(A,B,C)}(a,b,c)=w,
	\end{equation}
	then combine \eqref{pvw} and \eqref{gf}, we get an elliptic equation:
	\begin{equation} \mathcal{D}_{(A_0,B_0,C_0)}(a,b,c)+\{(a,b,c),(a,b,c)\}=(w,u,v)
		\label{ccc}
	\end{equation}
	where $(u,v)=(u_0,v_0)-\mathcal{RVW}(\tau^1,\tau^2,\tau^3,\theta_1,\theta_2,A_0,B_0,C_0)$. Equation \eqref{ccc} is called \emph{the perturbed   Vafa-Witten equation under Coulomb gauge}.

	Like the unperturbed case, we have the following regular result of the solutions of perturbed reduced Vafa-Witten equations.
	
	\begin{theorem} \label{tj} {\bf (Global estimate for $L_1^2$ solutions to the inhomogeneous perturbed reduced Vafa-Witten plus Coulomb slice equations, cf.  \cite[Corollary 3.4]{FL}, \cite[Theorem  3.1]{G}, \cite[Theorem 3.3.1]{Ma})\/}
		Let $(X,g)$ be a closed, oriented and smooth Riemannian 4-manifold, $P\to X$ a principle $SU(2)$- or $SO(3)$-bundle on $X$ and let $(A_0,B_0,C_0)$ be a $C^{\infty}$ configuration in $\mathcal{C}(P)$. Then there is a positive constant $\epsilon=\epsilon(A_0,B_0,C_0)$ such that if $(a,b,c)$ is an $L_1^2$ solution to the  equation \eqref{ccc}, where $(w,u,v)$ is in $L_k^2$ and $||(a,b,c)||_{L^4(X)}<\epsilon$, and $k\geq 3$ is an integer, then $(a,b,c)\in L_{k+1}^2$ and there is a polynomial $Q_k(x,y)$, with positive real coefficients, depending at most on $(A_0,B_0,C_0),k$ such that $Q_k(0,0)=0$ and
		$$||(a,b,c)||_{L_{k+1,A_0}^2(X)}\leq Q_k\Big(||(w,u,v)||_{L_{k,A_0}^2(X)},||(a,b,c)||_{L^2(X)}\Big).$$
		In particular, if $(w,u,v)$ is in $C^r$ then $(a,b,c)$ is in $C^{r+1}$. 
	\end{theorem}
	
	\begin{theorem} \label{tk}
		{\bf  (Global regularity of $L_k^2$ solutions to the perturbed reduced Vafa-Witten equations for $k\geq 3$, cf.  \cite[Proposition 3.7]{FL}, \cite[Theorem  3.2]{G}, \cite[Theorem 3.3.2]{Ma})} Let $(X,g)$ be a closed, oriented and smooth Riemannian 4-manifold and $P\to X$ a principle $SU(2)$- or $SO(3)$-bundle on $X$. Let $k\geq 3$ be an integer and suppose that $(A,B,C)$ is an $L_k^2$ solution to $\mathcal{RVW}(\tau^1,\tau^2,\tau^3,\theta_1,\theta_2,A,B,C)=0$ for fixed $C^r$ perturbation parameters $(\tau^1,\tau^2,\tau^3,\theta_1,\theta_2)$, then there is a gauge transformation $\zeta\in L_{k+1}^2(\mathcal{G}_P)$ such that $\zeta\cdot (A,B,C)$ is $C^\infty$ over $X$.
	\end{theorem}
	
	\section{Transversality of the reduced part of the moduli spaces}
	
	In this section we verify the perturbed reduced Vafa-Witten map \eqref{cd}, viewed as a section of the Banach vector bundle $\overline{\mathcal{E}}_k$ over $\mathcal{T}^r\times\mathcal{B}_k^\diamond(P)$, is transverse to the zero section of $\overline{\mathcal{E}}_k$ over $\mathcal{T}^r\times\mathcal{B}_k^\diamond(P)$. 
	
	\subsection{The linearization of the perturbed reduced Vafa-Witten map}
	
	Fix a perturbation parameter $(\tau^1,\tau^2,\tau^3,\theta_1,\theta_2)\in\mathcal{T}^r$, if $\Theta':=(\tau^1,\tau^2,\tau^3,\theta_1,\theta_2,A,B,C) \in \mathcal{T}^r\times\mathcal{B}_k(P)$ satisfies $\mathcal{RVW}(\Gamma)=0$, then the linearization of the map $\mathcal{RVW}$ at $\Gamma$ is 
	\begin{align*}
		&(D\mathcal{RVW})_{\Theta'}(\delta\tau^1,\delta\tau^2,\delta\tau^3,\delta\theta_1,\delta\theta_2,a,b,c)\\
		=&\begin{pmatrix}(D\mathcal{RVW}_1)_{\Theta'}(\delta\tau^1,\delta\tau^2,\delta\tau^3,\delta\theta_1,\delta\theta_2,a,b,c)\\
			(D\mathcal{RVW}_2)_{\Theta'}(\delta\tau^1,\delta\tau^2,\delta\tau^3,\delta\theta_1,\delta\theta_2,a,b,c)\end{pmatrix}\\
		=&\begin{pmatrix}
			d^*_Ab +d_A c-[B\centerdot a] -[C,a]+\delta\tau^1(B\centerdot\theta_1+[B\centerdot B]\centerdot\theta_2)+\tau^1(b\centerdot\theta_1+2[b\centerdot B]\centerdot\theta_2)\\
			+\tau^1(B\centerdot\delta\theta_1+[B\centerdot B]\centerdot\delta\theta_2)\\
			d_A^+a+\frac{1}{8}\delta\tau^1[B\centerdot B]+\frac{1}{4}\tau^1[b\centerdot B] +\frac12 [B,c] +\frac12 [b, C]+\delta\tau^2B+\tau^2b
		\end{pmatrix},
	\end{align*}
	where $(\delta\tau^1,\delta\tau^2,\delta\tau^3,\delta\theta_1,\delta\theta_2,a,b,c)\in\mathcal{T}^r\times T_{[A,B,C]}\mathcal{M}_{RVW}^*(\tau^1,\tau^2,\tau^3,\theta_1,\theta_2)$ (It's easy to see that $T_{(\tau^1,\tau^2,\tau^3,\theta_1,\theta_2)}\mathcal{T}^r=\mathcal{T}^r$).   For $\Theta:=(\tau^1,\tau^2,\tau^3,\theta_1,\theta_2,A,B,0) \in \mathcal{T}^r\times\mathcal{B}_k^\diamond(P)$, we have
	\begin{align*}
		&(D\mathcal{RVW})_{\Theta}(\delta\tau^1,\delta\tau^2,\delta\tau^3,\delta\theta_1,\delta\theta_2,a,b,c)\\
		=&\begin{pmatrix}
			d^*_Ab +d_A c-[B\centerdot a]+\delta\tau^1(B\centerdot\theta_1+[B\centerdot B]\centerdot\theta_2)+\tau^1(b\centerdot\theta_1+2[b\centerdot B]\centerdot\theta_2)\\
			+\tau^1(B\centerdot\delta\theta_1+[B\centerdot B]\centerdot\delta\theta_2)\\
			d_A^+a+\frac{1}{8}\delta\tau^1[B\centerdot B]+\frac{1}{4}\tau^1[b\centerdot B] +\frac12 [B,c] +\delta\tau^2B+\tau^2b
		\end{pmatrix}.
	\end{align*}
	Note that the full differential $(D\mathcal{RVW})_{\Theta}$ differs from the parameter fixed differential $d^1_{A,B,C}$ in \eqref{cg} by bounded linear terms in $\delta\tau^1,\delta\tau^2,\delta\tau^3,\delta\theta_1,\delta\theta_2$. This fact, together with the estimate in Theorem~\ref{tj}, implies that $(D\mathcal{RVW})_{\Theta}$ has closed range. Hence
	$$\mathrm{Ran}(D\mathcal{RVW})_{\Theta} \not= \mathcal{C}'_{k-1}(P)$$
	if and only if there is a nonzero pair $(\phi,\psi)\in\mathcal{C}'_{k-1}(P)=L_{k-1}^2(X,\mathfrak{su}(2)_P\otimes\Lambda^1)\oplus L_{k-1}^2(X,\mathfrak{su}(2)_P\otimes\Lambda^{2,+})$ such that $\forall(\delta\tau^1,\delta\tau^2,\delta\tau^3,\delta\theta_1,\delta\theta_2,a,b,c)\in\mathcal{T}^r\times T_{[A,B,0]}\mathcal{M}_{RVW}^*(\tau^1,$ $\tau^2,\tau^3,\theta_1,\theta_2)$ we have
	\begin{equation}\label{yy}\begin{aligned}
			&\langle(D\mathcal{RVW})_{\Theta}(\delta\tau^1,\delta\tau^2,\delta\tau^3,\delta\theta_1,\delta\theta_2,a,b,c),(\phi,\psi)\rangle_{L^2(X)}\\
			=&\langle(D\mathcal{RVW}_1)_{\Theta}(\delta\tau^1,\delta\tau^2,\delta\tau^3,\delta\theta_1,\delta\theta_2,a,b,c),\phi\rangle_{L^2(X)}\\
			+&\langle(D\mathcal{RVW}_2)_{\Theta}(\delta\tau^1,\delta\tau^2,\delta\tau^3,\delta\theta_1,\delta\theta_2,a,b,c),\psi\rangle_{L^2(X)}=0.
	\end{aligned}\end{equation}
	The above formula implies $(\phi,\psi)\in\mathrm{Ker}(D\mathcal{RVW})_{\Theta}^*$ \big(where $(D\mathcal{RVW})_{\Theta}^*$ is the $L^2(X)$ adjoint operator of $(D\mathcal{RVW})_{\Theta}$\big), then applying elliptic regularity for the Laplacian $(D\mathcal{RVW})_{\Theta}(D\mathcal{RVW})_{\Theta}^*$ with $C^{r-1}$ coefficients implies that $(\phi,\psi)$ is $C^{r+1}$(cf.  \cite[\S5]{FL}). And Aronszajn's theorem (cf.  \cite[Remark 3]{Ar}, \cite[Theorem 1.8]{Kaz}) implies that $(\phi,\psi)$ in $\mathrm{Ker}\big((D\mathcal{RVW})_{\Theta}(D\mathcal{RVW})_{\Theta}^*\big)$ has the unique continuation property (cf.  \cite[Lemma 5.9]{FL}). Therefore, to prove that $(\phi,\psi)\equiv 0$ on $X$, it is only necessary to prove that $(\phi,\psi)$ is zero on an open subset of $X$.
	
	\subsection{Transversality of the reduced part}
	
	Before establishing the transversality, we prove the following two lemmas first: The first lemma establishes the unique continuation property of the connection $A$ when $[A,B,0]$ is a zero point of the perturbed reduced Vafa-Witten map \eqref{cd}, the second lemma provides a full rank part of the first component of the map \eqref{cd}, which is crucial to the establishment of transversality. 
	
	As in \cite{G}, we need the following  Agmon-Nirenberg unique continuation theorem to prove Lemma \ref{xzz}.
	\begin{theorem}\label{agth}
		Let $(\mathfrak{H},(\cdot,\cdot))$ be a Hilbert space and let $\mathcal{P}:\mathrm{Dom}(\mathcal{P}(r))=\mathfrak{H}_D\subset\mathfrak{H}\to\mathfrak{H}$ be a family of symmetric linear operators for $r\in[r_0,R)$. Suppose that $\eta\in C^1([r_0,R),\mathfrak{H})$ with $\eta(r)\in\mathfrak{H}_D$ and $\mathcal{P}\eta\in C^0([r_0,R),\mathfrak{H})$ such that
		$$\left|\left|\frac{d\eta}{dr}-\mathcal{P}(r)\eta(r)\right|\right|\leq c_1||\eta(r)||$$
		for some positive constants $c_1$ and all $r\in[r_0,R)$. If the function $r\mapsto(\eta(r),\mathcal{P}(r)\eta(r))$ is differentiable for $r\in[r_0,R)$ and satisfies
		$$\left|\left|\frac{d\mathcal{P}}{dr}\eta\right|\right|\leq c_2(||\mathcal{P}\eta||+||\eta||)$$
		for positive constants $c_2$ and every $r\in[r_0,R)$, then the following holds: If $\eta(r)=0$ for $r_0\leq r\leq r_1<R$, then $\eta(r)=0$ for all $r\in[r_0,R)$.
	\end{theorem}
	
	\begin{lemma}\label{xzz}Let $(X,g)$ be an oriented, closed and smooth Riemannian 4-manifold and $P\to X$ a principle $SU(2)$- or $SO(3)$-bundle on $X$.  For every perturbation parameter $(\tau^1,\tau^2,\tau^3,\theta_1,\theta_2)\in\mathcal{T}^r$, if $[A,B,0]\in\mathcal{B}_k(P)$ is a solution to the perturbed reduced Vafa-Witten equation
		\begin{equation}\label{yz}\left\{
			\begin{aligned}
				&d_A^*B+d_AC+\tau^1(B\centerdot\theta_1+[B\centerdot B]\centerdot\theta_2)=0\\
				&F_A^++\frac{1}{8}\tau^2[B\centerdot B]+\frac{1}{2}[B,C]+\tau^3B=0\\
			\end{aligned}\right.\end{equation}
		such that $B\not\equiv 0$ and $[B\centerdot B]\equiv 0$ on $X$, then $A$ is reducible on $X$.
	\end{lemma}
	
	\begin{proof}
		For a solution $[A,B,0]\in\mathcal{B}_k(P)$ to \eqref{yz} such that $[B\centerdot B]\equiv 0$ on $X$, the equation \eqref{yz} can be reduced to
		\begin{equation}\label{yz2}\left\{
			\begin{aligned}
				&d_A^*B+\tau^1(B\centerdot\theta_1)=0,\\
				&F_A^++\tau^3B=0.\\
			\end{aligned}\right.\end{equation}
		
		For $\beta\in L_{k}^2(X,\mathfrak{su}(2)_P\otimes\Lambda^{2,+})$, define the linear map $$\mathcal{L}_{A,\tau^1,\theta_1}(\beta):=d_A^*\beta+\tau^1(\beta\centerdot\theta_1)\in L_{k-1}^2(X,\mathfrak{su}(2)_P\otimes\Lambda^1),$$
		then $\mathcal{L}_{A,\tau^1,\theta_1}(B)=0$ and also, $\mathcal{L}_{A,\tau^1,\theta_1}^*\mathcal{L}_{A,\tau^1,\theta_1}(B)=0$ where $\mathcal{L}_{A,\tau^1,\theta_1}^*$ is the $L^2$ adjoint of $\mathcal{L}_{A,\tau^1,\theta_1}$. So the Aronszajn's theorem (cf.  \cite[Remark 3]{Ar}, \cite[Theorem 1.8]{Kaz}) implies that $B$ has the unique continuation property. Let $X_B:=\{x\in X:B(x)\neq 0\}\subseteq X$, then $X_B$ is either $X$ or an open dense subset of $X$.
		
		Note that $[B\centerdot B]\equiv 0$ on $X$, $B$ is rank 1 on $X_B$ \cite[\S4.1.1]{Ma}, so there is $\xi\in \Omega^0(X_B,\mathfrak{su}(2)_P)$ with $\langle\xi,\xi\rangle=1$ and $\omega\in L_{k}^2(X,\Lambda^{2,+})$ such that $B=\xi\otimes\omega$ (cf. \cite[\S4.2]{Ma}). We have
		\begin{equation}
			\begin{aligned}\label{yz3}
				0&=d_A^*B+\tau^1(B\centerdot\theta_1)\\
				&=d_A^*(\xi\otimes\omega)+\tau^1(\xi\otimes\omega\centerdot\theta_1)\\
				&=-d_A\xi\centerdot\omega+\xi\otimes d^*\omega+\xi\otimes\tau^1(\omega\centerdot\theta_1).\\
			\end{aligned}
		\end{equation}
		$\langle\xi,\xi\rangle=1$ implies $\langle d_A\xi,\xi\rangle=0$, take inner product with $\xi$ we get $d^*\omega+\tau^1(\omega\centerdot\theta_1)=0$, so $d_A\xi\centerdot\omega=0$ on $X_B$ and $d_A\xi=0$ on $X_B$ (cf.  \cite[\S4.2]{Ma}). That's means $A$ is reducible on $X_B$.
		
		Next we will show that in fact $A$ is reducible on $X$. The method used here is analogous to the cases of the ASD equations \cite[Lemma 4.3.21]{D2} and $PU(2)$-monopole equations \cite[\S5.3.2]{FL}.
		
		Choose a point $x_0\in X_B$ and let $\rho$ be the injectivity radius of $X$ at $x_0$, there is a positive number $0<\epsilon<\frac{1}{2}\rho$ such that the geodesic ball $B(x_0,\epsilon)\subset X_B$. We will show that $A$ is reducible on $B(x_0,2\epsilon)$. We trivialize $\mathfrak{su}(2)_P$ over $B(x_0,2\epsilon)-\{x_0\}$ using parallel transport along radial geodesics (cf.  \cite[\S2.3.1]{D2}), this gives an isomorphism 
		\begin{equation}\label{is}
			\mathfrak{su}(2)_P|_{B(x_0,\epsilon)-\{x_0\}}\cong\mathfrak{su}(2)_{x_0}\times S^3\times(0,2\epsilon)
		\end{equation}
		where $\mathfrak{su}(2)_{x_0}\cong\mathfrak{su}(2)$. Note that $A$ is in radial gauge with respect to the point $x_0$, so the radial component of connection $A$ in this trivialization is zero. We let $\mathfrak{A}:=\mathfrak{A}(r)$, $r\in(0,2\epsilon)$, denote the resulting one-parameter family of the connection on the bundle $\mathfrak{su}(2)_{x_0}\times S^3$ over $S^3$. A section $B$ of the bundle $L^2_k(X,\mathfrak{su}(2)_P\otimes\Lambda^{2,+})$ over $B(x_0,2\epsilon)-\{x_0\}$ pulls back, via the isomorphism $\Lambda^{2,+}\otimes\mathfrak{su}(2)_P|_{B(x_0,2\epsilon)-\{x_0\}}\cong\Lambda^{2,+}\otimes\mathfrak{su}(2)_{x_0}\times S^3\times(0,2\epsilon)$, to a one-parameter family of sections $\mathfrak{B}(r)$ of the bundle $\Lambda^{2,+}\otimes\mathfrak{su}(2)_{x_0}\times S^3$ over $S^3$.
		
		Under the isomorphism \eqref{is}, $\xi$ can be viewed as a map to the structure group $SU(2)$ or $SO(3)$ and it satisfies
		$$0=\langle d_A\xi,\frac{\partial}{\partial r}\rangle=\frac{\partial\xi}{\partial r},$$
		hence we can extend $\xi$ by parallel
		translation via $A$ along radial geodesics emanating from $x_0$ to a gauge transformation $\hat{\xi}$ on connections on the bundle $\mathfrak{su}(2)_{x_0}\times S^3\times(0,2\epsilon)$. 
		
		Let $\hat{A}:=\hat{\xi}^{-1}\cdot A=A-(d_A\hat{\xi})\hat{\xi}^{-1}$, note that $d_A\hat{\xi}=0$ on $B(x_0,\epsilon)$, we have $\hat{A}=A$ on $B(x_0,\epsilon)$. Under the isomorphism $S^3\times(0,2\epsilon)\cong B(x_0,2\epsilon)-\{x_0\}$, the metric $g$ on $B(x_0,2\epsilon)-\{x_0\}$ pulls back to
		$$g=dr^2+g_r=dr^2+\gamma(r,\theta)d\theta^2,$$
		where $g_r:=\gamma(r,\theta)d\theta^2$ is the metric on $S^3$ pulled back from the restriction $g|_{S^3(x_0,r)}$ to the geodesic sphere $S^3(x_0,r):=\{x\in X:d_g(x,x_0)=r\}$. Denote by $*_{g_r}$ the Hodge star operator for the metric $g_r$ on $S^3$ and $*_g$ the Hodge star operator for the metric $g$ on $X$. For any differential form $\eta$ on $S^3$, we have 
		\begin{align*}
			*_g\eta&=dr\wedge*_{g_r}\eta,\\
			*_{g_r}\eta&=*_g(dr\wedge\eta).
		\end{align*}
		
		Let $\{e^1,e^2,e^3\}$ be an oriented, orthonormal frame for $T^*S^3$, then $\{e^0:=dr,e^1,e^2,e^3\}$ is an oriented orthonormal basis for $T^*X$ over $B(x_0,2\epsilon)-\{x_0\}$. Note that the linear space of self-dual $\R$-valued two forms on $X$ is 3-dimensional $C(X)$-linear space and $\R$-valued two forms on $S^3$ is 3-dimensional $C(S^3)$-linear space, hence a endomorphism $\tau^3\in C^r(\End_\R(\Lambda^{2,+}(X)))$ can also be regarded as an endomorphism of $L_k^2(S^3,\Lambda^2)$.
		
		With all above understood, $\{e^0\wedge e^1+e^2\wedge e^3,e^0\wedge e^2+e^3\wedge e^1,e^0\wedge e^3+e^1\wedge e^2\}$ is an orthonormal basis of $\Lambda^{2,+}(B(x_0,2\epsilon))$ (cf.  \cite[\S4.1.1]{Ma}). Let $A=\sum_{i=0}^3A_ie^i$ and $F_A=\sum_{0\leq i<j\leq 3}F_{ij}e^i\wedge e^j$ be the linear representation in the corresponding basis. The radial component of connection $A$ is zero implies $A_0=0$ on $B(x_0,2\epsilon)$ and we have $F_{0j}=\frac{dA_j}{dr}$, hence
		\begin{align*}
			F_A^+=&\frac{1}{2}(1+*_g)F_A\\
			=&\frac{1}{2}(F_{01}+F_{23})(e^0\wedge e^1+e^2\wedge e^3)+\frac{1}{2}(F_{02}+F_{31})(e^0\wedge e^2+e^3\wedge e^1)\\
			&+\frac{1}{2}(F_{03}+F_{12})(e^0\wedge e^3+e^1\wedge e^2)\\
			=&\frac{1}{2}e^0\wedge(F_{01}e^1+F_{02}e^2+F_{03}e^3)+\frac{1}{2}(F_{23}e^2\wedge e^3+F_{31}e^3\wedge e^1+F_{12}e^1\wedge e^2)\\
			=&\frac{1}{2}dr\wedge\frac{d\mathfrak{A}(r)}{dr}+\frac{1}{2}F_{\mathfrak{A(r)}}.
		\end{align*}
		
		Let $\tau^3=\begin{pmatrix}
			\tau^3_{11} & \tau^3_{12} & \tau^3_{13}\\
			\tau^3_{21} & \tau^3_{22} & \tau^3_{23}\\
			\tau^3_{31} & \tau^3_{32} & \tau^3_{33}\\
		\end{pmatrix}$ be the representation matrix under the basis $\{e^0\wedge e^1+e^2\wedge e^3,e^0\wedge e^2+e^3\wedge e^1,e^0\wedge e^3+e^1\wedge e^2\}$, it's not hard to see that the representation matrix of $\tau^3$ under the basis $\{e^2\wedge e^3, e^3\wedge e^1,e^1\wedge e^2\}$ of $L_k^2(S^3,\Lambda^2)$ is the same one. Let
		$$\omega=B_1(e^0\wedge e^1+e^2\wedge e^3)+B_2(e^0\wedge e^2+e^3\wedge e^1)+B_3(e^0\wedge e^3+e^1\wedge e^2)$$
		where $B_i\in L_k^2(B(x_0,2\epsilon),\R)$, $i=1,2,3$, then 
		$$\mathfrak{B}(r)=\xi\otimes(B_1e^2\wedge e^3+B_2e^3\wedge e^1+B_3e^1\wedge e^2)$$
		and we have
		\begin{align*}
			\tau^3B=&\tau^3(\xi\otimes\omega)\\
			=&\xi\otimes\big(B_1\tau^3(e^0\wedge e^1+e^2\wedge e^3)+B_2\tau^3(e^0\wedge e^2+e^3\wedge e^1)+B_3\tau^3(e^0\wedge e^3+e^1\wedge e^2)\big)\\
			=&\xi\otimes\Big(B_1\big(\tau^3_{11}(e^0\wedge e^1+e^2\wedge e^3)+\tau^3_{12}(e^0\wedge e^2+e^3\wedge e^1)+\tau^3_{13}(e^0\wedge e^3+e^1\wedge e^2)\big)+\\
			&B_2\big(\tau^3_{21}(e^0\wedge e^1+e^2\wedge e^3)+\tau^3_{22}(e^0\wedge e^2+e^3\wedge e^1)+\tau^3_{23}(e^0\wedge e^3+e^1\wedge e^2)\big)+\\
			&B_3\big(\tau^3_{31}(e^0\wedge e^1+e^2\wedge e^3)+\tau^3_{32}(e^0\wedge e^2+e^3\wedge e^1)+\tau^3_{33}(e^0\wedge e^3+e^1\wedge e^2)\big)\Big)\\
			=&\xi\otimes e^0\wedge\big((B_1\tau^3_{11}+B_2\tau^3_{21}+B_3\tau^3_{31})e^1+(B_1\tau^3_{12}+B_2\tau^3_{22}+B_3\tau^3_{32})e^2\\
			&+(B_1\tau^3_{13}+B_2\tau^3_{23}+B_3\tau^3_{33})e^3\big)+\xi\otimes\big((B_1\tau^3_{11}+B_2\tau^3_{21}+B_3\tau^3_{31})e^2\wedge e^3\\
			&+(B_1\tau^3_{12}+B_2\tau^3_{22}+B_3\tau^3_{32})e^3\wedge e^1+(B_1\tau^3_{13}+B_2\tau^3_{23}+B_3\tau^3_{33})e^1\wedge e^2\big)\\
			=&dr\wedge*_{g_r}\tau^3\mathfrak{B}(r)+\tau^3\mathfrak{B}(r).
		\end{align*}
		So $F_A^++\tau^3B=0$ is equivalent to
		$$dr\wedge\left(\frac{1}{2}\frac{d\mathfrak{A}(r)}{dr}+*_{g_r}\tau^3\mathfrak{B}(r)\right)+\left(\frac{1}{2}F_{\mathfrak{A(r)}}+\tau^3\mathfrak{B}(r)\right)=0,$$
		which means
		\begin{equation}
			\left\{
			\begin{aligned}
				\frac{1}{2}\frac{d\mathfrak{A}(r)}{dr}+*_{g_r}\tau^3\mathfrak{B}(r)&=0\\
				\frac{1}{2}F_{\mathfrak{A(r)}}+\tau^3\mathfrak{B}(r)&=0\\
			\end{aligned}\right.
		\end{equation}
		and hence on $B(x_0,2\epsilon)$,
		\begin{equation}\label{asd3}
			\frac{d\mathfrak{A}(r)}{dr}=*_{g_r}F_{\mathfrak{A(r)}}.
		\end{equation}
		$\hat{\mathfrak{A}}(r)$ and $\mathfrak{A}(r)$ are both solutions to \eqref{asd3}, and $\hat{\mathfrak{A}}(r)=\mathfrak{A}(r)$ when $0<r\leq\epsilon$. Let $a(t):=\hat{\mathfrak{A}}(r)-\mathfrak{A}(r)$, then $a(r)=0$ when $0<r\leq\epsilon$ and
		\begin{equation}
			\begin{aligned}\label{ode}
				\frac{da}{dr}&=*_{g_r}\left(d\hat{\mathfrak{A}}(r)-d\mathfrak{A}(r)+\frac{1}{2}[\hat{\mathfrak{A}}(r),\hat{\mathfrak{A}}(r)]-\frac{1}{2}[\mathfrak{A}(r),\mathfrak{A}(r)]\right)\\
				&=*_{g_r}\left(da+\frac{1}{2}[\mathfrak{A},a]+\frac{1}{2}[a,\hat{\mathfrak{A}}]\right).
			\end{aligned}
		\end{equation}
		Hence
		$$\left|\left|\frac{da}{dr}-*_{g_r}da\right|\right|_{L_k^2}\leq c_1||a||_{L_k^2}$$
		for some constant $c_1$ depending on $\hat{\mathfrak{A}}(r)$ and $\mathfrak{A}(r)$. 
		
		On $\Omega^1(S^3)$, $*_{g_r}^2=(-1)^{1*(3-1)}=1$. Hence for any two 1-forms $a_1,a_2\in L_k^2((S^3,g_r),\Lambda^1)$, we have
		\begin{align*}
			&\int_{S^3}\langle a_1,*_{g_r}da_2\rangle d\mathrm{vol}_r\\
			=&\int_{S^3}\langle a_1,*_{g_r}d*_{g_r}^2a_2\rangle d\mathrm{vol}_r=\int_{S^3}\langle a_1,(-1)^{3(2+1)+1}d^{*_{g_r}}*_{g_r}a_2\rangle d\mathrm{vol}_r\\
			=&\int_{S^3}\langle da_1,*_{g_r}a_2\rangle d\mathrm{vol}_r=\int_{S^3}\langle*_{g_r}da_1,a_2\rangle d\mathrm{vol}_r,
		\end{align*}
		which means the operators $*_{g_r}d$ are self-adjoint with respect to the metrics $g_r$, $r\in(0,2\epsilon)$.
		Let $d\mathrm{vol}$ be the volume form on $S^3$ defined by the standard metric $g_{std}$, then
		$$d\mathrm{vol}_r=h_r^2d\mathrm{vol},~r\in(0,2\epsilon)$$
		for some positive function $h_r$ on $S^3$. Define the operators $\mathcal{Q}_r:=h_r\circ(*_{g_r}d)\circ h_r^{-1}$ for every $r\in(0,2\epsilon)$ and Hilbert-space isomorphisms $L_k^2((S^3,g_r),\Lambda^1)\to L_k^2((S^3,g_{std}),\Lambda^1)$ by $a\to\alpha:=h_ra$, then
		\begin{align*}
			&\int_{S^3}\langle \mathcal{Q}_r(\alpha_1),\alpha_2\rangle d\mathrm{vol}=\int_{S^3}\langle h_r\circ(*_{g_r}d)\circ h_r^{-1}h_ra_1,h_ra_2\rangle d\mathrm{vol}\\
			=&\int_{S^3}\langle *_{g_r}da_1,a_2\rangle d\mathrm{vol}_r=\int_{S^3}\langle a_1,*_{g_r}da_2\rangle d\mathrm{vol}_r\\
			=&\int_{S^3}\langle h_ra_1,h_r\circ(*_{g_r}d)\circ h_r^{-1}h_ra_2\rangle d\mathrm{vol}=\int_{S^3}\langle \alpha_1,\mathcal{Q}_r(\alpha_2)\rangle d\mathrm{vol},
		\end{align*}
		therefore $\mathcal{Q}_r$ is self-adjoint with respect to $g_{std}$ for every $r\in(0,2\epsilon)$ and has dense domain $L_1^2((S^3,g_{std}),\Lambda^1)$. Since $a=h_r^{-1}\alpha$, we have
		$$\frac{da}{dr}=-h_r^{-2}\frac{dh_r}{dr}\alpha+h_r^{-1}\frac{d\alpha}{dr}$$
		and subsituting into \eqref{ode} gives
		\begin{equation}
			\begin{aligned}
				-h_r^{-2}\frac{dh_r}{dr}\alpha+h_r^{-1}\frac{d\alpha}{dr}&=h_r^{-1}\mathcal{Q}_rh_ra+\frac{1}{2}*_{g_r}([\mathfrak{A},h_r^{-1}\alpha]+[h_r^{-1}\alpha,\hat{\mathfrak{A}}]),\\
				\frac{d\alpha}{dr}&=\mathcal{Q}_r\alpha+\frac{1}{2}*_{g_r}([\mathfrak{A},\alpha]+[\alpha,\hat{\mathfrak{A}}])+h_r^{-1}\frac{dh_r}{dr}\alpha
			\end{aligned}
		\end{equation}
		for $r\in(0,2\epsilon)$. On compact interval $[\frac{1}{2}\epsilon,2\epsilon]$, the continuous function $h_r^{-1}\frac{dh_r}{dr}$ is bounded and hence there is a constant $c_1$ independent of $a$ such that
		$$\left|\left|\frac{d\alpha}{dr}-\mathcal{Q}_r\alpha\right|\right|_{L_k^2}\leq c_1||\alpha||_{L_k^2},~\frac{1}{2}\epsilon\leq r\leq2\epsilon.$$
		Also, the definition of $\mathcal{Q}_r$ yields the pointwise bounds
		$$\left|\left(\frac{d\mathcal{Q}_r}{dr}a\right)(r)\right|\leq c'(|(\nabla a)(r)|+|a(r)|)$$
		for $\frac{1}{2}\epsilon\leq r\leq2\epsilon$ and some constant $c'$ independent of $a$, where $\nabla$  denotes covariant derivatives on $\Lambda^{1}\otimes\mathfrak{su}(2)_{x_0}\times S^3$ over $S^3$ which is independent of $r$. Thus the standard elliptic estimate for $\mathcal{Q}_r$ implies
		$$\left|\left|\frac{d\mathcal{Q}_r}{dr}a\right|\right|_{L_k^2}\leq c_2(||\mathcal{Q}_ra||_{L_k^2}+||a||_{L_k^2})$$
		for $\frac{1}{2}\epsilon\leq r\leq2\epsilon$ and some constant $c_2$ independent of $a$. The conditions of Theorem \ref{agth} have been verified and $a(r)=0$ when $r\in(0,\epsilon)$, so we must have $a(r)=0$ for $r\in(0,2\epsilon)$, which means $d_A\hat{\xi}=0$ on $B(x_0,2\epsilon)$.
		
		By applying the above argument to every point $x_0$ of the open dense subset $X_B\subset X$, we can see that the extension $\hat{\xi}$ of $\xi$ can be extended to all of $X$ and $d_A\hat{\xi}=0$ on $X$, hence $A$ is reducible on $X$.
	\end{proof}
	
	\begin{lemma}\label{xz}Let $(X,g)$ be an oriented, closed and smooth Riemannian 4-manifold and $P\to X$ a principle $SU(2)$- or $SO(3)$-bundle on $X$. Then there is a first-category subset $\mathcal{T}^r_1\subset\mathcal{T}^r$ such that for every perturbation parameter $(\tau^1,\tau^2,\tau^3,\theta_1,\theta_2)\in\mathcal{T}^r-\mathcal{T}^r_1$, if $[A,B,0]\in\mathcal{B}_k^\diamond(P)$ is a solution to the perturbed reduced Vafa-Witten equation
		\eqref{yz}, then $B\centerdot\theta_1+[B\centerdot B]\centerdot\theta_2$ is rank 3 on some open subset $U\subset X$.
	\end{lemma}

	\begin{proof}
		Define $\mathcal{T}^r_1:=\{(\tau^1,\tau^2,\tau^3,\theta_1,\theta_2)\in\mathcal{T}^r:\theta_1\wedge\theta_2=0$ on some open subset of $X\}$, it's not hard to see that $\mathcal{T}^r_1$ is a nowhere dense, hence first-category subset of $\mathcal{T}^r$. Fix any $(\tau^1,\tau^2,\tau^3,\theta_1,\theta_2)\in\mathcal{T}^r-\mathcal{T}^r_1$ and 
		let $[A,B,0]\in\mathcal{B}_k^\diamond(P)$ be a solution to \eqref{yz}. The definition of $\mathcal{B}_k^\diamond(P)$ implies that $A$ is irreducible on $X$, hence by Lemma \ref{xzz} we see that $[B\centerdot B]\not\equiv0$ on $X$, so there is an open subset $U\subset X$ such that $[B\centerdot B]\neq 0$ on $U$. 
		
		Choose local orthogonal basis $\{\sigma_1,\sigma_2,\sigma_3\}$ of $L_{k}^2(U,\Lambda^{2,+})$ and $\{\eta_1,\eta_2,\eta_3\}$ of $\mathfrak{su}(2)_P$ with $\sigma_1\centerdot\sigma_2=-2\sigma_3$, $[\eta_1,\eta_2]=2\eta_3$ and cyclic permutations on $U$ (cf. Lemma~\ref{xxy}) such that $B=B_1\eta_1\otimes\sigma_1+B_2\eta_2\otimes\sigma_2+B_3\eta_3\otimes\sigma_3$, $B_i\in L_k^2(U,\R)$, $i=1,2,3$, cf.  \cite[\S4.1.1]{Ma}. Then
		\begin{align*}
			[B\centerdot B]=&-8B_2B_3\eta_1\otimes\sigma_1-8B_3B_1\eta_2\otimes\sigma_2-8B_1B_2\eta_3\otimes\sigma_3,\\
			B\centerdot\theta_1+[B\centerdot B]\centerdot\theta_2=&\eta_1\otimes\sigma_1\centerdot(B_1\theta_1-8B_2B_3\theta_2)+\eta_2\otimes\sigma_2\centerdot(B_2\theta_1-8B_3B_1\theta_2)\\
			&+\eta_3\otimes\sigma_3\centerdot(B_3\theta_1-8B_1B_2\theta_2).
		\end{align*}
		Note that $[B\centerdot B]\neq 0$ on $U$, without loss of generality we assume $B_2B_3\neq 0$ on $U$.
		If the rank of $B\centerdot\theta_1+[B\centerdot B]\centerdot\theta_2$ on $U$ is less than 3, then
		one of $\eta_1\otimes\sigma_1\centerdot(B_1\theta_1-8B_2B_3\theta_2)$, $\eta_2\otimes\sigma_2\centerdot(B_2\theta_1-8B_3B_1\theta_2)$,
		$\eta_3\otimes\sigma_3\centerdot(B_3\theta_1-8B_1B_2\theta_2)$
		is zero on $U$, Lemma~\ref{xxy} shows that this is equivalent to one of $B_1\theta_1-8B_2B_3\theta_2,B_2\theta_1-8B_3B_1\theta_2,B_3\theta_1-8B_1B_2\theta_2$ is zero on $U$. But
		\begin{align*}
			B_1\theta_1-8B_2B_3\theta_2=0&\Rightarrow\theta_2=\frac{B_1}{8B_2B_3}\theta_1,\\
			B_2\theta_1-8B_3B_1\theta_2=0&\Rightarrow\theta_1=\frac{8B_3B_1}{B_2}\theta_2,\\
			B_3\theta_1-8B_1B_2\theta_2=0&\Rightarrow\theta_1=\frac{8B_1B_2}{B_3}\theta_2
		\end{align*}
		on $U$ and anyone of the three equalities implies $\theta_1\wedge\theta_2=0$ on $U$, contradicts to the choice of $\theta_1$ and $\theta_2$. Hence $B\centerdot\theta_1+[B\centerdot B]\centerdot\theta_2$ is rank 3 on $U$.
	\end{proof}
	
	Now we prove that $(\phi,\psi)$ is zero on an open subset of $X$. For a perturbation parameter $(\tau^1,\tau^2,\tau^3,\theta_1,\theta_2)\in\mathcal{T}^r-\mathcal{T}^r_1$, let $[A,B,0]\in\mathcal{B}_k^\diamond(P)$ be a solution to the corresponding perturbed reduced Vafa-Witten equation \eqref{yz}. Lemma \ref{xzz} implies that there is an open subset $V\subset X$ such that $[B\centerdot B]\neq 0$ on $V$. In equations \eqref{yy}, set $(\delta\tau^1,\delta\tau^3,\delta\theta_1,\delta\theta_2,a,b,c)=0$ we have
	$$\langle\delta\tau^2[B\centerdot B],\psi\rangle_{L^2(X)}=0,~\forall\delta\tau^2\in C^r(\mathfrak{gl}(\Lambda^{2,+})).$$
	Similiarly,
	$$\langle\delta\tau^3B,\psi\rangle_{L^2(X)}=0,~\forall\delta\tau^3\in C^r(\mathfrak{gl}(\Lambda^{2,+})).$$
	On $V$, $[B\centerdot B]\neq 0$ implies the rank of $B$ is at least 2\cite[\S4.1.1]{Ma}, so Lemma~\ref{xx} and  \cite[Lemma 2.3]{Fe} imply $\psi\equiv 0$ on $V$. Then set $(\delta\tau^2,\delta\tau^3,\delta\theta_1,\delta\theta_2,a,b,c)=0$ we can get
	$$\delta\tau^1(B\centerdot\theta_1+[B\centerdot B]\centerdot\theta_2)=0,~\forall\delta\tau^1\in C^r(\mathfrak{gl}(\Lambda^1))$$
	on $V$, then Lemma~\ref{xz} and  \cite[Lemma 2.3]{Fe} imply $\phi\equiv 0$ on $V$.
	
	Hence we must have $(\phi,\psi)\equiv 0$ on $V$, so $(\phi,\psi)\equiv 0$ on $X$ by the unique continuation for the Laplacian $(D\mathcal{RVW})_{\Theta}(D\mathcal{RVW})_{\Theta}^*$ and we finish the establishment of the transversality.
	
	By the slice result (cf.  \cite[Proposition 2.8]{FL}), $T_{[A,B,0]}\mathcal{B}_k^\diamond(P)$ may be identified with $\mathrm{Ker}~d_{(A,B,0)}^{0,*}$. For $(a,b,c)\in\mathrm{Ker}~d_{(A,B,0)}^{0,*}$, we have
	$$\begin{aligned}
		(D\mathcal{RVW})_{(A,B,0)}(0,1,0,0,0,a,b,c)&=d_{(A,B,0)}^1(a,b,c)\\
		&=(d_{(A,B,0)}^{0,*}+d_{(A,B,0)}^1)(a,b,c),\end{aligned}$$
	and it's easy to see that the differential $(D\mathcal{RVW})_{(A,B,0)}|_{\{0\}\times T\mathcal{C}_k^\diamond(P))}$ is Fredholm, where ${\{0\}\times T\mathcal{C}_k^\diamond(P)}=T((\tau^1,\tau^2,\tau^3,\theta_1,\theta_2)\times\mathcal{B}_k^\diamond(P))$. Thus $\mathcal{RVW}$ is a Fredholm section when restricted to the fixed-parameter fibers $(\tau^1,\tau^2,\tau^3,\theta_1,\theta_2)\times\mathcal{C}_k^\diamond(P)\subset(\mathcal{T}^r-\mathcal{T}^r_1)\times\mathcal{C}_k^\diamond(P)$ where $(\tau^1,\tau^2,\tau^3,\theta_1,\theta_2)\in\mathcal{T}^r-\mathcal{T}^r_1$, so the Sard-Smale theorem (cf. \cite[Proposition 4.12]{Fe}, \cite[Proposition 4.3.11]{D2}) implies that there is a first-category subset $\mathcal{T}^r_2\subset\mathcal{T}^r-\mathcal{T}^r_1$ such that the zero sets in $\mathcal{C}_k^\diamond(P)$ of $\mathcal{RVW}(\tau^1,\tau^2,\tau^3,\theta_1,\theta_2,\cdot)$ are regular (note the transversality) for all perturbations $(\tau^1,\tau^2,\tau^3,\theta_1,\theta_2)\in\mathcal{T}^r-(\mathcal{T}^r_{1}\cup\mathcal{T}^r_{2})$. 
	
	In summary, we have the following theorem, which is also Theorem \ref{main}.
	
	\begin{theorem}Let $(X,g)$ be a closed, oriented and smooth Riemannian 4-manifold, Then there is a first-category subset $\mathcal{T}^r_{fc}\subset\mathcal{T}^r$ such that for all $(\tau^1,\tau^2,\tau^3,\theta_1,\theta_2)$ in $\mathcal{T}^r-\mathcal{T}^r_{fc}$ the following holds: The zero set of the section $\mathcal{RVW}(\tau^1,\tau^2,\tau^3,\theta_1,\theta_2,\cdot)$ in $\mathcal{C}_k^\diamond(P)$ is regular and the moduli space $\mathcal{M}_{RVW}^*(\tau^1,\tau^2,\tau^3,\theta_1,\theta_2)=\big(\mathcal{RVW}(\tau^1,\tau^2,\tau^3,\theta_1,\theta_2,\cdot)^{-1}(0)/\mathcal{G}_{k+1}(P)\big)\cap\mathcal{B}_k^\diamond(P)$ is a smooth manifold of dimension 0.
	\end{theorem}

	\ni\textbf{Acknowledgements.} 
	This research is partially supported by NSFC grants 12201255. The author is grateful to Bo Dai, Huijun Fan and Yuuji Tanaka for invaluable help in the production of this paper.
	
	\section*{Appendix}
	\appendix
	\setcounter{theorem}{0}
	\renewcommand\thetheorem{A.\arabic{theorem}}
	
	In this appendix, we explain some notation, and prove a few technical lemmas.
	
	Let $V$ be a finite dimensional inner product space. Let $\{ e_1,\cdots, e_n\}$ be an orthonormal basis for $V$, and  $\{e^1,\cdots,e^n\}$  the dual basis for $V^*$. For any $\alpha\in\Lambda^{p}V^*$, $\beta\in\Lambda^{q}V^*$, we define
	$$\alpha\centerdot\beta=(-1)^{p-1}\sum_{i=1}^n(\iota_{e_i}\alpha)\wedge(\iota_{e_i}\beta)\in\Lambda^{p+q-2}V^*,$$
	where $\iota_{e_i}$ is the contraction with $e_i$.
	
	The exterior algebra $\L^\bullet V^*$ inherits an inner product, such that $\{ e^{i_1}\wedge e^{i_2}\wedge\cdots\wedge e^{i_{p}}; 0\leq p\leq n, i_1 <i_2 <\cdots <i_p\}$ forms an orthonormal basis.  The inner product of two elements $\a,\b \in \L^\bullet V^*$ is denoted as $\a\cdot\b$,

	Replacing $V$ by the tangent space $T_xX$ at any point $x\in X$ of a Riemannian manifold $X$, we can define the products $\centerdot$ and $\cdot$ for differential forms. Let $\mathfrak{g}$ be the  Lie algebra of a compact Lie group $G$, equipped with an invariant inner product $\langle\cdot,\cdot\rangle$, here invariance means $\langle [\xi,\eta],\zeta\rangle=\langle \xi,[\eta,\zeta]\rangle$, $\forall \xi,\eta,\zeta \in\mathfrak{g}$. If $\mathfrak{g}$ is $\mathfrak{su}(2)$, an invariant inner product is given by $\langle \xi,\eta\rangle=-\frac{1}{2}\mathrm{tr}(\xi\eta)$ \cite[(A.18)]{Ma}, where $\xi,\eta\in\mathfrak{su}(2)$ are regarded as matrices.
	
	Let $P\to X$ be a principal $G$-bundle and $\fg_P$ be the adjoint bundle. For $\fg_P$-valued forms, for example, if $\omega_1=\xi_1\otimes\alpha_1$ and $\omega_2=\xi_2\otimes\a_2$, where $\xi_1,\xi_2\in\W^0(X, \fg_P)$ and $\alpha_1\in\Omega^p(X)$, $\a_2\in\Omega^q(X)$, we define $[\omega_1\centerdot\omega_2] =[\xi_1,\xi_2]\otimes\alpha_1\centerdot\a_2$ and $[\omega_1\cdot\omega_2]=[\xi_1,\xi_2]\otimes\alpha_1\cdot\a_2$, $\langle\omega_1\cdot \omega_2\rangle=\langle \xi_1, \xi_2\rangle \a_1\cdot \a_2$. See  \cite[Appendix]{Ma} for more detail.
	
	Over a Riemannian 4-manifold $X$, the \emph{rank} of a section $B\in\Omega^{2,+}(X, \fg_P)$ is defined as follows. Choose local frames for $\fg_P$ and  $\Lambda^{2,+}(T^*X)$, then the section $B$ is represented by a $d\times 3$ matrix-valued function with respect to the local frames, where $d=\dim G$. The rank of $B$ at a point of $X$ is just the rank of the matrix at that point, and $\rank(B)$ is the maximum of the pointwise rank over $X$. If $\mathfrak{g}$ is $\mathfrak{su}(2)$, then the maximum of the rank of $B$ is $d=3$.
	
	\begin{lemma}\label{xx}Let $\alpha$, $\beta$ be two elemmaents of $\mathfrak{su}(2)$ and they are linearly independent, then $\alpha,\beta,[\alpha,\beta]$ form a basis of the linear space $\mathfrak{su}(2)$.
	\end{lemma}
	
	\begin{proof} Let $\alpha=\alpha_1\eta_1+\alpha_2\eta_2+\alpha_3\eta_3$ and $\beta=\beta_1\eta_1+\beta_2\eta_2+\beta_3\eta_3$ where $\alpha_i,\beta_i\in\mathbb{R}$, $i=1,2,3$, $\{\eta_1,\eta_2,\eta_3\}$ is a  basis for $\mathfrak{su}(2)$ such that $[\eta_1,\eta_2]=2\eta_3$ and cyclic permutations. Then we have
		$$[\alpha,\beta]=2(\alpha_2\beta_3-\alpha_3\beta_2)\eta_1+2(\alpha_3\beta_1-\alpha_1\beta_3)\eta_2+2(\alpha_1\beta_2-\alpha_2\beta_1)\eta_3.$$
		We need to show that $\alpha,\beta,[\alpha,\beta]$ are linearly independent. Let $k_1,k_2,k_3\in\mathbb{R}$ be three real numbers and $$k_1\alpha+k_2\beta+k_3[\alpha,\beta]=0,$$
		then note that $\{\eta_1,\eta_2,\eta_3\}$ are linearly independent, we must have
		$$\left\{
		\begin{aligned}
			k_1\alpha_1+k_2\beta_1+2k_3(\alpha_2\beta_3-\alpha_3\beta_2)=0\\
			k_1\alpha_2+k_2\beta_2+2k_3(\alpha_3\beta_1-\alpha_1\beta_3)=0\\
			k_1\alpha_3+k_2\beta_3+2k_3(\alpha_1\beta_2-\alpha_2\beta_1)=0.\\
		\end{aligned}\right.$$
		The determinant of the above linear system of equations is
		$$\begin{vmatrix}
			
			\alpha_1 & \beta_1 & 2(\alpha_2\beta_3-\alpha_3\beta_2)\\
			
			\alpha_2 & \beta_2 & 2(\alpha_3\beta_1-\alpha_1\beta_3)\\
			
			\alpha_3 & \beta_3 & 2(\alpha_1\beta_2-\alpha_2\beta_1)\\
			
		\end{vmatrix}=2(\alpha_2\beta_3-\alpha_3\beta_2)^2+2(\alpha_3\beta_1-\alpha_1\beta_3)^2+2(\alpha_1\beta_2-\alpha_2\beta_1)^2.$$
		Note that $\alpha$, $\beta$ are linearly independent, so $\alpha\neq 0$, $\beta\neq 0$ and there's not a nonzero real number $k\in\mathbb{R}$ such that $\alpha=k\beta$, which means $\alpha_2\beta_3-\alpha_3\beta_2,\alpha_3\beta_1-\alpha_1\beta_3,\alpha_1\beta_2-\alpha_2\beta_1$ cannot be all zero, so the determinant is greater than zero, which implies $k_1=k_2=k_3=0$, so $\alpha,\beta,[\alpha,\beta]$ are linearly independent and they form a base of the three-dimensional linear space $\mathfrak{su}(2)$.
	\end{proof}
	
	\begin{lemma}\label{xxy}Let $0\neq B\in\mathfrak{su}(2)\otimes\Lambda^{2,+}\mathbb{R}^4$, $\theta\in\Lambda^{1}\mathbb{R}^4$ and $B\centerdot\theta=0$, then $\theta=0$.
	\end{lemma}
	\begin{proof}
		According to the singular value decomposition for $\mathfrak{su}(2)\otimes\Lambda^{2,+}\mathbb{R}^4$ in  \cite[\S 4.1.1]{Ma}, there exist oriented orthonormal basis $\{e^1,e^2,e^3,e^4\}$ for $\mathbb{R}^4$ and a basis $\{\eta_1,\eta_2,\eta_3\}$ for $\mathfrak{su}(2)$ such that $B=B_{1}\eta_1\otimes(e^1\wedge e^2+e^3\wedge e^4)+B_{2}\eta_2\otimes(e^1\wedge e^3+e^4\wedge e^2)+B_{3}\eta_3\otimes(e^1\wedge e^4+e^2\wedge e^3)$, 
		where $B_{1},B_{2},B_{3}\in\R$. And let $\theta=\theta_1e^1+\theta_2e^2+\theta_3e^3+\theta_4e^4$ where $\theta_k\in\R$, $k=1,2,3,4$. Then
		\begin{align*}
			&B\centerdot\theta\\
			=&-B_1\theta_1\eta_1\otimes e^2-B_2\theta_1\eta_2\otimes e^3-B_3\theta_1\eta_3\otimes e^4+B_1\theta_2\eta_1\otimes e^1\\
			&+B_2\theta_2\eta_2\otimes e^4-B_3\theta_2\eta_3\otimes e^3
			-B_1\theta_3\eta_1\otimes e^4+B_2\theta_3\eta_2\otimes e^1\\
			&+B_3\theta_3\eta_3\otimes e^2+B_1\theta_4\eta_1\otimes e^3-B_2\theta_4\eta_2\otimes e^2+B_3\theta_4\eta_3\otimes e^1.
		\end{align*}
		The linear independence of $\{\eta_i\otimes e^j\}_{1\leq i\leq 3,1\leq j\leq 4}$ implies $B_i\theta_j=0$, $\forall 1\leq i\leq 3,1\leq j\leq 4$, and note that $B\neq 0$, at least one of $B_1,B_2,B_3$ is nonzero, hence $\theta_j=0$ for $j=1,2,3,4$ and $\theta=0$.
	\end{proof}

\end{document}